\documentclass[12pt]{article}
\usepackage{times}
\usepackage{amsmath,amsfonts,amstext,amssymb,amsbsy,amsopn,amsthm,eucal}
\usepackage{txfonts}
\usepackage{dsfont}
\usepackage{graphicx}   

\numberwithin{equation}{section}
  \theoremstyle{plain}
 \newtheorem{theorem}[equation]{Theorem}
 
\newtheorem{proposition}[equation]{Proposition}
 \newtheorem{lemma}[equation]{Lemma}
 \newtheorem{corollary}[equation]{Corollary}

 \theoremstyle{remark}

 \newtheorem{remark}[equation]{Remark}

\theoremstyle{definition}
 \newtheorem{definition}[equation]{Definition}

\newcommand{\ux}{\underline{x}}
\newcommand{\Ric}{{\rm Ric}}
\newcommand{\Vol}{{\rm Vol}}

\newcommand{\dN}{\mathds{N}}
\newcommand{\dR}{\mathds{R}}

\newcommand{\dZ}{\mathds{Z}}

\newcommand{\R}{{ \rm R}}

\newcommand{\cB}{\mathcal{B}}
\newcommand{\cC}{\mathcal{C}}
\newcommand{\cD}{\mathcal{D}}

\newcommand{\cM}{\mathcal{M}}
\newcommand{\cN}{\mathcal{N}}
\newcommand{\cR}{\mathcal{R}}
\newcommand{\cS}{\mathcal{S}}

\newcommand{\cV}{\mathcal{V}}
\newcommand{\cW}{\mathcal{W}}

\newcommand{\tv}{{\rm v}}

\begin{document}

\title{Lower Bounds on Ricci Curvature\\and\\Quantitative Behavior of Singular Sets}

\author{Jeff Cheeger\thanks{The author was partially supported by NSF grant DMS1005552 } and Aaron Naber\thanks{The author was partially supported by NSF postdoctoral grant 0903137}}
\date{}
\maketitle
\begin{abstract}
Let $Y^n$ denote the Gromov-Hausdorff limit 
$M^n_i\stackrel{d_{\rm GH}}{\longrightarrow} Y^n$ of $\tv$-noncollapsed
 riemannian manifolds  with ${\rm Ric}_{M^n_i}\geq-(n-1)$.
 The singular set $\cS\subset Y$ has a stratification
$\cS^0\subset\cS^1\subset\cdots\subset\cS$, where
 $y\in \cS^k$ if
no tangent cone at $y$ splits off a factor $\dR^{k+1}$ isometrically.
  Here, we define for all $\eta>0$, $0<r\leq 1$, the
{\it $k$-th effective singular stratum} $\cS^k_{\eta,r}$ satisfying
$\bigcup_\eta\bigcap_r \,\cS^k_{\eta,r}= \cS^k$.
Sharpening
the known Hausdorff dimension bound
$\dim\, \cS^k\leq k$, we prove that for all $y$,
the volume of  the $r$-tubular neighborhood of $\cS^k_{\eta,r}$
 satisfies
$\Vol(T_r(\cS^k_{\eta,r})\cap B_{\frac{1}{2}}(y))\leq c(n,\tv,\eta)r^{n-k-\eta}$.
The proof involves a  {\it quantitative differentiation} argument.
This result has applications to Einstein manifolds.
Let $\cB_r$ denote the set of points at which
the $C^2$-harmonic radius is $\leq r$.
If also the $M^n_i$ are K\"ahler-Einstein with $L_2$ curvature bound,
$||Rm||_{L_2}\leq C$, then
$\Vol( \cB_r\cap B_{\frac{1}{2}}(y))\leq c(n,\tv,C)r^4$ for all $y$. In
the K\"ahler-Einstein  case,  without assuming
any  integral curvature bound on the $M^n_i$, we obtain  a slightly weaker
 volume bound on $\cB_r$ which yields
 an  a priori $L_p$ curvature bound for all $p<2$.
 \end{abstract}
{\small \tableofcontents}
\pagebreak

\section{Volume bounds for quantitative singular sets}
\label{s:sr}
Let $(M^n,g)$ denote a riemannian manifold whose Ricci curvature satisfies
\begin{align}
\label{con:Lower_Ricci}
\Ric_{M^n}\geq -(n-1)g\, .
\end{align}
Let $\Vol_{-1}(r)$ denote the volume of a ball of
radius $r$ in $n$-dimensional hyperbolic space of curvature $\equiv -1$.
We will assume $M^n$ is $\tv$-noncollapsed i.e. for all $x\in M^n$,
\begin{align}
\label{con:Noncollapsed}
\frac{\Vol(B_1(x))}{\Vol_{-1}(1)}\geq \tv>0\, .
\end{align}

Let
\begin{equation}
\label{ghl}
M^n_i\stackrel{d_{\rm GH}}{\longrightarrow} Y^n\, ,
\end{equation}
denote the Gromov-Hausdorff limit
(possibly in the pointed sense)
 of a sequence of manifolds $M^n_i$
 satisfying (\ref{con:Lower_Ricci}),
(\ref{con:Noncollapsed}). In this case, the measured Gromov-Hausdorff
limit of the riemannian measures on the $M^n_i$ is \hbox{$n$-dimensional} Hausdorff
measure on
$Y^n$. We will simply denote it  by $\Vol(\, \cdot\,)$.

Relations (\ref{con:Lower_Ricci})--(\ref{ghl}) will be in force throughout
the paper.

 For $y\in Y^n$, every tangent
cone $Y_y$ is a metric cone $C(Z)$ with cross-section $Z$ and vertex $z^*$.
  A point $y$ is called {\it regular} if one (equivalently, every) tangent
cone is isometric to $\R^n$. Otherwise $y$ is called {\it singular}.  The set of
singular points is denoted by $\cS$. The stratum $\cS^k\subset \cS$ is defined
as the set of points for which no tangent cone splits off  isometrically a factor
$\R^{k+1}$. In fact, $\cS^{n-1}\setminus\cS^{n-2}=\emptyset$. Thus,
\begin{equation}
\label{strat}
\cS^0\subset\cS^1\subset\cdots\subset\cS^{n-2}=\cS\, .
\end{equation}
Moreover, in the sense of Hausdorff dimension, we have
\begin{equation}
\label{hd}
\dim \, \cS^k\leq k\, .
\end{equation}
For the all of the above, see \cite{cheegercoldingonthestruct1}.

Given $\eta>0$, $0< r<1$, we will define a quantitative version $\cS^k_{\eta,r}$
 of the singular stratum $\cS^k$. The criterion for membership of $y\in Y^n$ in
$\cS^k_{\eta,r}$ involves the behavior of $B_s(y)$ for all $r\leq s\leq 1$.
We will
show that  $\Vol(\cS^k_{\eta,r})\leq c(n,\tv,\eta)r^{n-k-\eta}$; see Theorem
\ref{t:S^k_Quantitative_Est}.

\begin{remark}
\label{r:sc}
In the special case, $Y^n=M^n$, with $M^n$ smooth, the sets $\cS^k$ are empty for
all $k$.
However, the sets $\cS^k_{\eta,r}$ need not be  empty.  In fact, in the proofs of all
of the estimates stated in this section, we can (and will) restrict attention to the
case
of smooth manifolds.
Since the measure on $Y^n$ is the
limit of the riemannian measures on the $M^n_i$, once proved for smooth manifolds,
the estimates
 pass immediately to Gromov-Hausdorff limit spaces.
\end{remark}

Denote by $(\underline{0},z^*)$, a vertex of the metric cone with isometric
splitting $\R^{k+1}\times C(Z)$. 

\begin{definition}
\label{d:strat}
 For $\eta>0$ and $0<r<1$, define the {\it $k$th effective singular stratum}
$\cS^k_{\eta,r}\subseteq Y^n$ by
$$
\cS^k_{\eta,r}:=\{y  \,\, |\,\,
d_{GH}(B_s(y),B_s((\underline{0},z^*)))
\geq\eta s, \,\,{\rm for\,\, all}\,\, \dR^{k+1}\times C(Z)\, \, {\rm and\,\,
all}\,\, r\leq s\leq 1\}\, .
$$
\end{definition}
It follows directly from the definition that
\begin{equation}
\label{containments}
\cS^{k}_{\eta,r}\subset \cS^{k'}_{\eta',r'}\qquad
({\rm if}\,\,k'\leq k,\,\,\eta'\leq\eta,\,\, r\leq r')\, .
\end{equation}
Also, if $y\in\cS^k$, then clearly, $y\in\bigcap_r\cS^k_{\eta,r}$, for some
$\eta>0$, so
\begin{equation}
\label{relation}
\cS^k=\bigcup_\eta \,\,\bigcap_r \, \cS^k_{\eta,r}\, .
\end{equation}

 Our first main result is a  volume bound for $\cS^k_{\eta,r}$.
The proof
will proceed by appropriately bounding the number of balls of radius $r$ needed to
cover
$\cS^k_{\eta,r}\cap B_1(\ux)$. Since by volume comparison, we have $\Vol(B_r(x))\leq
c(n)r^n$,
 so this will suffice.
\begin{theorem}
\label{t:S^k_Quantitative_Est}
There exists $c(n,\tv,\eta)>0$ such that if  $ M^n_i\stackrel{d_{\rm
GH}}{\longrightarrow} Y^n$,
and the $(M^n_i,g_i)$ satisfy the lower Ricci curvature bound
(\ref{con:Lower_Ricci}), and
$\tv$-noncollapsing condition (\ref{con:Noncollapsed}), then for all $y\in Y^n$ and
$\eta>0$,
\begin{align}
\label{mainve}
\Vol(\cS^k_{\eta,r}\cap B_{\frac{1}{2}}(y))\leq c(n,\tv,\eta)r^{n-k-\eta}.
\end{align}
\end{theorem}

\begin{remark}
It is an easy consequence of the definition
of $\cS^k_{\eta,r}$, that the bound in (\ref{mainve})
actually implies (for a slightly different constant $c(n,\tv,\eta)$)
\begin{equation}
\label{tubebound}
\Vol(T_r(\cS^k_{\eta,r}) \cap B_{\frac{1}{2}}(y)))\leq c(n,\tv,\eta)r^{n-k-\eta}\, ,
\end{equation}
 where $T_r(\cS^k_{\eta,r}) $
denotes the $r$-tubular neighborhood.
\end{remark}

\begin{remark}
There is a possibility that that on the right-hand side of (\ref{mainve}), the factor
$r^{n-k-\eta}$ can be replaced by one of the form $r^{n-k}(\log r)^{c(n,\tv,\eta)}$.
However, it seems unlikely that in general, the appearance of $\eta$ on the
right-hand side
of (\ref{mainve}) can be entirely removed. In the application to K\"ahler-Einstein
manifolds given in Theorem \ref{t:Effective_Curvature_Estimate2}, this
is of no consequence since the bound in (\ref{mainve})
controls a lower order term; compare
(\ref{curvestp}).
\end{remark}

\begin{remark}
As will be indicated in Sections \ref{s:outline},
\ref{s:ip}, the proof of Theorem \ref{t:S^k_Quantitative_Est} employs an instance of
{\it quantitative differentiation} in the sense of Section 14 of \cite{CKN}.
\end{remark}



\begin{definition}
\label{hrdef}
If $y\in Y^n$ and the metric is not $C^2$ in some neighborhood of $y$,
then  $r_{har}(y)=0$. Otherwise,
 $r_{har}(x)$ is  the supremum of
those $r$ such that the ball $B_r(y)$ is contained in the domain of a harmonic
coordinate system such that $g_{ij}(0)=\delta_{ij}$ and
$$
|g_{ij}|_{C^1}\leq r^{-1}\, ,\qquad |g_{ij}|_{C^2}\leq r^{-2}\, .
$$
\end{definition}
Recall that by elliptic regularity, there exist constants, $c(n,k)$
such that if
$M^n$ is Einstein, then for $g_{ij}$ as above, we have 
$$|g_{ij}|_{C^k}\leq c(n,k)r_{har}^{-k}\, .$$
Also, the curvature tensor $Rm$ satisfies
\begin{equation}
\label{hb}
\sup_{B_{r_{har}(y)}(y)}|Rm|\leq c(n)r_{har}^{-2}\, .
\end{equation}

Put
\begin{equation}
\label{cBdef}
\cB_r=\{y\, |\, r_{har}(y)\leq r\}\, .
\end{equation}

\begin{remark}
 Let $\widetilde\cB_r\subset Y^n$ denote the set of points such that either
$r_0(y)=0$ or $|Rm(y)|\geq c(n)r^{-2}$. In particular, $\widetilde\cB_r\subset \cB_r$.
Since $r_{har}$ is $1$-Lipschitz, it follows that
\begin{equation}
\label{tube2}
T_r(\widetilde\cB_r)\subseteq T_r(\cB_r)\subseteq \cB_{2r}\, .
\end{equation}
  Hence estimates
on $\Vol(\cB_r)$ imply estimates on $\Vol(T_r(\widetilde\cB_r))$. Equivalently,
in view of the noncollapsing assumption (\ref{con:Noncollapsed}),
estimates
on $\Vol(\cB_r)$  imply on the covering
number of $\widetilde\cB_r$.
\end{remark}

Under the additional assumption that the $M^n_i$ are Einstein and (for some of our
results)
 satisfy an integral curvature bound, we will apply Theorem
\ref{t:S^k_Quantitative_Est}
in combination with $\epsilon$-regularity theorems to control the
volume of the set $\cB_r$. In this case, we replace (\ref{con:Lower_Ricci}) by the
2-sided bound
\begin{align}
\label{con:Bounded_Ricci}
|\Ric_{M^n_i}|\leq n-1\, .
\end{align}
Our first  result of this type
follows  by combining Theorem \ref{t:S^k_Quantitative_Est}
with the $\epsilon$-regularity theorems, Theorem 6.2 of
\cite{cheegercoldingonthestruct1}
and Theorem 5.2 of \cite{MR1978491}\footnote{The latter is due independently to G.
Tian.};
the detailed argument is given in Section
\ref{s:Effective_Curvature_Estimates1}.

\begin{theorem}
\label{t:Effective_Curvature_Estimate1}
There exists $\eta_0=\eta_0(n,\tv)>0$ such that if $ M^n_i\stackrel{d_{\rm
GH}}{\longrightarrow} Y^n$,
and the $M^n_i$ are Einstein manifolds satisfying the $\tv$-noncollapsing
condition (\ref{con:Noncollapsed}), and the Ricci curvature bound
(\ref{con:Bounded_Ricci}),  then for every $0<r<1$:
\begin{enumerate}
\item If $\eta<\eta_0$, then we have $\cB_r\subset \cS^{n-2}_{\eta,r}\, .$
In particular, for all $y\in Y^n$,
\begin{equation}
\label{volest1}
\Vol(\cB_{r}\cap B_{\frac{1}{2}}(y))\leq c(n,\tv,\eta)r^{2-\eta}\, .
\end{equation}
\item If in addition, the $M^n_i$ are K\"ahler, then
we have $\cB_r\subset \cS^{n-4}_{\eta,r}$.
In particular, for all $y\in Y^n$,
\begin{equation}
\label{volest2}
\Vol(\cB_{r}\cap B_{\frac{1}{2}}(y))\leq c(n,\tv,\eta)r^{4-\eta}\, .
\end{equation}
\end{enumerate}
\end{theorem}

\begin{remark}
Conjecturally, in item 2. above, the K\"ahler assumption can be dropped.
\end{remark}

\begin{corollary}
\label{c:L^p_Curvature_Bound}
Let  $Y^n$ be as in Theorem \ref{t:Effective_Curvature_Estimate1}. Then:
 \begin{enumerate}
\item In case 1. of Theorem \ref{t:Effective_Curvature_Estimate1}, for every $0<p<1$,
$$
\fint_{B_{\frac{1}{2}}(y)}|Rm|^{p}\leq c(n)\cdot
\fint_{B_{\frac{1}{2}}(y)}(r_{har})^{-2p} <
c(n,\tv,p)
\qquad ({\rm for\,\,all}\,\,p<1)\, .
$$
\item In case 2. of Theorem \ref{t:Effective_Curvature_Estimate1},  for every $0<p<2$,
$$
\fint_{B_{\frac{1}{2}}(y)}|Rm|^{p}\leq c(n)\cdot
\fint_{B_{\frac{1}{2}}(y)}(r_{har})^{-2p}<
c(n,\tv,p)
\qquad ({\rm for\,\,all}\,\,p<2)\, .
$$
\end{enumerate}
\end{corollary}
\begin{remark}
Theorem \ref{t:Effective_Curvature_Estimate1}
and Corollary \ref{c:L^p_Curvature_Bound} remain true
 assuming the Ricci curvature bound (\ref{con:Bounded_Ricci}) and
a bound on $|\nabla \Ric_{M^n}^i|$.  Alternatively,
If the $C^2$-harmonic radius is replaced by the
$C^{1,\alpha}$-harmonic radius, then
Theorem  \ref{t:Effective_Curvature_Estimate1}
and Corollary \ref{c:L^p_Curvature_Bound} hold with only the Ricci curvature bound,
$|\Ric_{M^n_i}|\leq n-1$.
\end{remark}

\begin{remark}
Even if $\cB_r$ were replaced by the smaller set $\widetilde\cB_r$, the assertions of
Corollary \ref{c:L^p_Curvature_Bound} would be new.
\end{remark}

In our next result (whose proof will be given in Section
\ref{s:Effective_Curvature_Estimates2}) we assume in addition, the  $L_p$ curvature
bound
\begin{align}
\label{con:L^p_Curvature_Bound}
\fint_{B_{1}(x)}|Rm|^p\leq C\, .
\end{align}
Recall in this connection, that for  K\"ahler-Einstein manifolds, we have the
topological
$L_2$ curvature bound
\begin{equation}
\label{kechbound}
\int_{M^n}|Rm|^2\leq c(n)\cdot\left(|(c^2_1\cup [\omega]^{(n/2)-2})(M^n)|
+|(c_2\cup[\omega]^{(n/2)-2})(M^n) | \right)\, ,
\end{equation}
where $c_1,c_2$ denote the first and second Chern classes and $[\omega]$
denotes the K\"ahler class; see e.g. \cite{MR1937830} and compare also the $L_p$
 bound ($p<2$) in item 2. of Corollary \ref{c:L^p_Curvature_Bound}, which holds
without assuming a bound on the right-hand side of (\ref{kechbound}).

\begin{theorem}
\label{t:Effective_Curvature_Estimate2}
\footnote{The results of
 the present paper arose in the course
of our ongoing investigations concerning the structure of Gromov-Hausdorff limit spaces
with Ricci curvature bounded below and in particular, on the structure of the
singular set
for limits of Einstein manifolds.  On the other hand, it has come to our attention
that Theorem \ref{t:Effective_Curvature_Estimate2} and Remark \ref{r:bigcomponent}
below
are stated as  conjectures (Hypothesis V and Supplements)
in an informal document  ``Discussion of the K\"ahler-Einstein problem''
written by S. Donaldson in the summer of 2009, available at
http://www2.imperial.ac.uk/~skdona/KENOTES.PDF.
It was announced there that the complex dimension 3 case of
Theorem \ref{t:Effective_Curvature_Estimate2}
would be treated in a
 forthcoming paper of Donaldson and X.  Chen; see \cite{chendonaldson1}.
The general case is treated in \cite{chendonaldson2}. Their work, like ours, makes use
of  \cite{cheegercoldingalmostrigidity}, \cite{cheegercoldingonthestruct1},
\cite{MR1937830}.  Unlike ours, it utilizes essentially
 a rigidity result for almost complex structures; see \cite{chakrabartishaw} .}
Let the assumptions be as in Theorem \ref{t:Effective_Curvature_Estimate1} and
assume in addition that the $M^n_i$ are K\"ahler-Einstein and satisfy the $L_p$
curvature bound (\ref{con:L^p_Curvature_Bound}), for some integer $p$, with
$2\leq p\leq \frac{n}{2}$.  Then for every $0<r<1$,
\begin{equation}
\label{curvestp}
\Vol(\cB_{r}\cap B_{\frac{1}{2}}(y))\leq c(n,\tv,C)r^{2p}\, .
\end{equation}
In particular, if the right-hand side of (\ref{kechbound}) is
bounded by $C$, then
\begin{equation}
\label{curvestch}
\Vol(\cB_{r}\cap B_{\frac{1}{2}}(y))\leq c(n,\tv,C)r^4\, .
\end{equation}
\end{theorem}

It is of key importance that $\eta$  does not appear on the right-hand side of
(\ref{curvestp}), (\ref{curvestch}); compare (\ref{hd}), (\ref{mainve}),
 (\ref{volest1}), (\ref{volest2}). Let us indicate how this
comes about.

Note that the estimates in (\ref{curvestp}), (\ref{curvestch}),
strengthen the known bounds on the Hausdorff measure $\mathcal H^{n-2p}(\cS^{n-2p})$
which in particular is finite; see \cite{MR1937830}, \cite{MR1978491}.
Those bounds are obtained
by combining standard maximal function estimates for the $L_p$ norm of the curvature
with the certain \hbox{$\epsilon$-regularity} theorems to estimate
 $\mathcal H^{n-2p}(\cS^{n-2p}\setminus \cS^{n-2p-1})$, and then using
(\ref{hd}), $\dim\, \cS^{n-2p-1}$\linebreak
$\leq n-2p-1$, which implies
$\mathcal H^{n-2p}(\cS^{n-2p-1})=0$.

 In fact, a slight modification of the first part of the argument
 gives the leading term on the right-hand side of
 (\ref{curvestch}), whereas the terms controlled by
Theorem \ref{t:S^k_Quantitative_Est}, which are lower order,
can be (and are) suppressed. The bound on these terms
(which requires the hypothesis of Theorem \ref{t:Effective_Curvature_Estimate2})
 can be viewed
as strengthened version of the estimate $\mathcal H^{n-2p}(\cS^{n-2p-1})=0$.

 More specifically,
 Theorem
\ref{t:S^k_Quantitative_Est} is only used to control the volumes of certain
subsets of $\cS^{n-2p-1}_{\eta_0,\gamma^{-i}}$, where $r\leq \gamma^{-i}\leq 1$,
 $1>\eta_0=\eta_0(n)>0$ is  sufficiently small and $\gamma=\gamma(\eta_0)$.
(The precise meaning of ``sufficiently small'' is dictated by the
constant in the $\epsilon$-regularity theorem of Section 5 of \cite{MR1978491}.)
 For instance, in the extreme case in which $\gamma^{-(i+1)}\leq r$,
we have
$\Vol(\cS^{n-2p-1}_{\eta_0,r})\leq c(n,\tv,\eta_0)r^{2p+1-\eta_0}$ and the sum of
the remaining
terms satisfies a bound of the same form.
Since $2p+1-\eta_0>2p$, the volume bound on these terms
can be suppressed.

\begin{remark}
 In the proof of Theorem \ref{t:Effective_Curvature_Estimate1}
by contrast,
Theorem \ref{t:S^k_Quantitative_Est} is used to control the highest order term.
\end{remark}

\begin{remark}
It is possible that Theorem \ref{t:Effective_Curvature_Estimate2} holds for Einstein
manifolds
 which are not necessarily K\"ahler. In any case, if $p$ is an even integer, then
apart from
some exceptional cases,  the $\epsilon$-regularity theorems of Section 8 can be used to
 show that (\ref{curvestp}) holds.  For $p$ an integer, using Section 4 of
\cite{MR1978491},
one gets (with no exceptional cases and all $\eta>0$) the less sharp estimate
$$
\Vol(\cB_{r}\cap B_{\frac{1}{2}}(y))\leq
c(n,\tv,\eta,C)r^{2p-\eta}\, .
$$
 \end{remark}

\begin{remark}
\label{r:bigcomponent}
Among the connected components of \hbox{$B_{\frac{1}{2}}(y)\setminus \cB_r$},
there is a component $\hat A_r$, such that
\begin{equation}
\label{e:connectedness}
\Vol(B_{\frac{1}{2}}(y)\setminus \hat A_r)\leq c(n,\tv,C)r^{\frac{(2p-1)n}{n-1}}\, .
\end{equation}
 To see this note that as previously mentioned, $B_{\frac{1}{2}}(y)\setminus
\cB_r\subset \cC_r$,
for some subset $\cC_r$ which is the union of at most $c(n,\tv,C)r^{-2p}$
balls of radius $r$. Moreover, for $r=r_0(n,\tv,\eta)$ sufficiently small,
there exists
$B_{r_0}(y')\subset \left(B_{\frac{1}{2}}(y)\setminus \cC_{r_0}\right)$.
 For $r\leq r_0$, let $A_r\subset \left(B_{\frac{1}{2}}(y)\setminus \cC_{r}\right)$
denote the component containing $B_{r_0}(y')$.
Clearly, $\Vol_{n-1}(\partial \cC_r)\leq c(n,\tv,C)r^{2p-1}$ and in particular,
$\Vol_{n-1}(\partial A_r)\leq c(n,\tv,C)r^{2p-1}$.
Since $\Vol(A_r)\geq \Vol(B_{r_0}(y'))$ (a definite lower bound) the isoperimetric
inequality for manifolds satisfying (\ref{con:Lower_Ricci}),
(\ref{con:Noncollapsed}), gives
$\Vol(B_{\frac{1}{2}}(y)\setminus A_r)\leq c(n,\tv,C)r^{\frac{(2p-1)n}{n-1}}$.
This implies (\ref{e:connectedness}).
\end{remark}

\begin{remark}
As briefly indicated in the discussion following the statement of Theorem
\ref{t:Effective_Curvature_Estimate2},
in proving that theorem,
 the  $\epsilon$-regularity theorem must be applied on all scales between
$1$ and $r$. Here, the fact that
 the hypothesis of the relevant $\epsilon$-regularity requires that
 two distinct conditions must be satisfied {\it simultaneously}  raises an
 issue that does not arise
in the proof of Theorem \ref{t:Effective_Curvature_Estimate1};
for details,  see Section \ref{s:Effective_Curvature_Estimates2}.
\end{remark}

\section{Outline of the proof of Theorem \ref{t:S^k_Quantitative_Est}}
\label{s:outline}

Since the methodology used in proving Theorem \ref{t:S^k_Quantitative_Est}
is new and is applicable in many other contexts (see e.g.
\cite{cnminhar})
we will give an informal explanation of the main ideas.

To prove Theorem \ref{t:S^k_Quantitative_Est}, we exhibit
$\cS^k_{\eta,r}$ as a  generalized Cantor set.  In particular,
we show that at most locations and scales $\geq r$,
there exists $\ell\leq k$, such that
$\cS^k_{\eta,r}$  lies very close to a $k$-dimensional subset of
the form  $\R^\ell\times \{z^*\}$, where $\R^\ell$ is a
factor  of an approximate local isometric splitting.
 Once this has been done, the volume computation is an essentially
 standard induction argument
based on iterated ball coverings.

 The following toy example illustrates our approach
in  highly  simplified situation corresponding to  the case $\cS^0$.
Notably,  a significant issue which must be
addressed in the actual situation
is not present in the toy example; see the subsection below entitled
``Implementation of cone-splitting''.

Start with the interval $[0,1]$ (so in effect, we are pretending
that $n=1$, although this plays no essential role).
Remove a subinterval from the center, then remove central subintervals
from each of the two remaining subintervals, etc.. Fix $\eta>0$.
We chose the lengths
of the $2^i$ distinct subintervals which remain at the $i$-th stage to be
$r_i=t_1\cdots t_i$, where we assume that for some $i(\eta)<\infty$, we have
$t_i^\eta\leq \frac{1}{2}$
for all $j>i(\eta)$.
Denote the generalized Cantor set which is intersection of this sequence of subsets
by $C$. The following volume estimate strengthens the
Hausdorff dimension estimate $\dim \, C\leq \eta$.

 Set
$ \max_{i\leq i(\eta)}2^i r_{i}^\eta= c(\eta)$.
Then  $2^jr_j^\eta\leq c(\eta)$ for any $j\geq i(\eta)$.
 For all $j$, we  have
$$
\begin{aligned}
\Vol(T_{r_j}(C))&\leq 2^j\cdot r_j\\
                         &\leq (2^j\cdot r_j^\eta)\cdot r_j^{1-\eta}\\
&\leq c(\eta)r_j^{1-\eta}\, ,
\end{aligned}
$$
which easily implies the same estimate with $r_j$ replaced
by any $r\leq 1$ and $c(\eta)$ replaced by $2\cdot c(\eta)$.

\subsection{The inequality $\dim\, \cS^k\leq  k$. }
Next we recall from  \cite{cheegercoldingonthestruct1}, the proof of the inequality
$\dim\, \cS^k\leq k$.
The proof  relies on an iterated blow up argument.
 The  following geometric facts are used.
 i) For limit spaces satisfying (\ref{con:Lower_Ricci}),
(\ref{ghl}), every
 tangent cone $Y_y$ is a metric cone. ii) The splitting theorem holds for
such
tangent cones (even for those which are collapsed).

Consider first the case $k=0$. By a density argument, if
$\dim\, \cS^0=0$
were to fail, it would already fail for some tangent cone
$Y_y$ for which the vertex
$y_\infty^*$, is a density point of $ \cS^0(Y_y)$. Thus, there would exist
$y'_\infty\in \cS^0(Y_y)$, a density point of $ \cS^0(Y_y)$, with $y'_\infty\ne
y_\infty^*$.
Moreover, by same the reasoning,  the assertion
would  fail in the same way
for some tangent cone $(Y_y)_{y'_\infty}$ at $y'_\infty$.
But since $y'_\infty\ne y_\infty^*$, by i), $y'_\infty$ is an interior point of
a ray emanating  from  $y_\infty^*$. After blow up at $y'_\infty$,
we obtain  a line
in $(Y_y)_{y'_\infty}$ and a density point $(y'_\infty)_\infty$ of
$\cS_0$ lying on this line. By ii), this line splits off isometrically, which
contradicts
$(y'_\infty)_\infty\in \cS_0$.
Similarly, by employing
additional blow ups, one gets $\dim\, \cS^k\leq k$ for all $k$; for further details,
see
\cite{cheegercoldingonthestruct1}.

\subsection{An issue involving multiple scales.}
Proving Theorem
\ref{t:S^k_Quantitative_Est} requires either finding a quantitative version of the
preceeding (noneffective) blow up argument,
or finding a different argument which can in fact be made quantitative. It is natural
to investigate the following idea for quantifying the blow up argument:
 Rather than doing
multiple blow ups to split off additional lines as isometric factors,
do an  ``appropriate'' sequence of rescalings which stop   short of going to the
blow up  the limit.
 The difficulty is that this leads to a sequence balls whose
radii decrease very rapidly and the resulting issue of having
to work simultaneously on a sequence of different scales.
 In fact,
it is not clear to us how to resolve the quantitative issues which arise from this approach.

Instead of blow up we use a different principle, the ``cone-splitting principle''. 
When its hypotheses are satisfied, the cone-splitting gives rise  to an ``additional
splitting''
of a single cone on a fixed scale. We show that in our context,
the hypotheses are indeed satisfied at most locations and scales.
In particular, this gives a new proof that $\dim\, \cS^k\leq k$ (though of course,
the quantitative version that we actually prove is much stronger).

\subsection{Cone-splitting, a replacement for  blow up.}

In its nonquantitative form, the cone-splitting principle gives a criterion
which guarantees that a metric cone $\R^\ell\times C(Z)$, which splits off
a Euclidean factor $\R^\ell$, actually splits off a factor of $\R^{\ell+1}$.
(Here and in the next paragraph, all splittings are isometric and
 $C(Z)$ denotes a metric
cone with vertex $z^*$.)
\vskip2mm

\noindent
{\bf Cone-splitting:}
Suppose that for some $C(\underline{Z})$
 with vertex $\underline{z}^*$, there is an isometry
$I:\R^\ell\times C(Z)\to C(\underline{Z})$ such that
$\underline{z}^*\not\in I(\R^\ell\times \{z^*\})$.  Then for some $W$,
 $\R^\ell\times C(Z)$ is isometric to $R^{\ell+1}\times C(W)$.\footnote{For our
purposes, we only need the cone-splitting principle
for tangent cones,  which case it
follows from the splitting theorem of \cite{cheegercoldingalmostrigidity}.
In fact, by an elementary argument (which we omit)
the cone-splitting principle holds for arbitrary metric cones. We do not
know an explicit reference for this fact.}

To see the relevance, note that in the proof of $\dim \cS^0=0$ which
was recalled above, if we knew that $y'_\infty\neq y$ was the vertex of some other
cone structure on $Y_y$, then $Y_y\equiv \dR\times Y'_y$.  Thus, we would obtain the
required ``additional splitting'' without the necessity of passing to a blow up.  In
actuality, we need the following quantitative version,
which is stated somewhat informally; for the precise statement,
see  Lemma \ref{l:qs}.
\vskip2mm

\noindent
{\bf Quantitative version of the cone-splitting principle:}
Consider a metric ball $B_r(p)$ and
 for $\delta=\delta(\eta)$
sufficiently small, a $\delta r$-Gromov-Hausdorff
equivalence $J_\delta: B_r(p)\to B_r((\underline{0},z^*)\subset
\R^\ell\times C(Z)$. Also assume  for some $q\in B_r(p)$, that
$J_\delta(q)$ does not lie too close to $J_\delta( \R^\ell\times \{z^*\})$,
Finally, assume that there is a $\delta r$-Gromov-Hausdorff equaivalence
$J'_\delta:B_r(q)\to B_r(\underline{z}^*)\subset C(\underline{Z})$.
Then $B_r(p)$ is $\eta r$-Gromov-Hausdorff
close to a ball $B_r((\underline{0},w^*))\subset\R^{\ell+1}\times  C(W)$,
for some cone $ \R^{\ell+1}\times C(W)$.

\subsection{Implementation of cone-splitting.}

As noted above, if we knew that $y'_\infty$ was the vertex of some
(other) cone structure on $Y_y$, then
we would obtain the required ``additional splitting'' without the necessity
of passing to a blow up. Roughly speaking, to implement the quantitative version of
cone-splitting,  we need
to know that this holds approximately at most locations and scales.

In fact, given a suitable notion of scale, $\gamma<1$, then for each $x$,
the balls $B_{\gamma^i}(x)$
($i=0,1,\ldots$) look as conical as we like (with $x$ playing the role of the vertex)
on all but a definite number of scales
$\gamma^i$.
This statement, which is close to being implicit in \cite{cheegercoldingonthestruct1},
is a quantitative version of the fact that tangent cones
are metric cones. It
constitutes  a  ``quantitative differentiation'' theorem  in the sense of  Section
14 of
\cite{CKN}.

Were it not for the fact that the collection of excluded scales
(those scales $\gamma^i$ for which $B_{\gamma^i}(x)$ is not sufficiently close to looking conical)
might depend on the point $x$,
we could use the cone-splitting principle to show that $\cS^k_{\eta,\gamma^j}$
``looks as $k$-dimensional as we like'' on all but a definite number of scales.
Since there is a bound on the number of excluded scales this easily suffices to
complete the Cantor type volume computation. This amounts to inductively
bounding the number of balls of radius $\gamma^j$ needed to cover
$\cS^k_{\eta,\gamma^j}$.
The general volume bound for  $\cS^k_{\eta,r}$, ${}$ $0<r\leq 1$, follows directly
from the
case $r=\gamma^j$.

In order to deal with
the above mentioned difficulty, we decompose the space into  subsets,
each of which consists of those points with precisely the same collection of excluded
scales. The bound on the number of excluded scales
 has the additional consequence that there are ``not to many'' of these subsets.
To each such set, we can apply the argument based on cone-splitting. Since there
``not to many'' such sets, we can simply add the resulting estimates. This finishes
the proof.
(Without bringing in this decomposition, we do not know how to complete
the argument.)

\section{Reduction to the covering lemma}
\label{s:ip}

As noted at the beginning of Section \ref{s:sr}, in proving Theorem
\ref{t:S^k_Quantitative_Est},
we can
(and will) restrict attention to the case of smooth manifolds.
Suppose for some convenient choice
$\gamma=\gamma(\eta)<1$, we can prove (\ref{mainve}) with some
constant $\tilde c(n,\tv,\eta)$ and
all $r$ of the form $\gamma^j$. Given $r$ arbitrary,
by choosing $j$ such that $\gamma^{j+1}< r\leq \gamma^j$,
we obtain (\ref{mainve}) for this $r$ with constant
$c(n,\tv,\eta)=\tilde c(n,\tv,\eta)(\gamma(\eta))^{-(n-k-\eta)}$.
Thus, in proving (\ref{mainve}), we can (and will) consider only $r$ of the form
$\gamma^j$.

An appropriate choice of $\gamma$ is given in (\ref{e:fixgamma}).
Lemma \ref{l:reduction} below (the covering lemma) asserts that the the set
$\cS^k_{\eta,\gamma^j}$ can be covered by a collection of sets,
$\{\cC^k_{\eta,\gamma^j}\}$, each of which consists of a not too large collection of
balls of radius $\gamma^j$.
 The cardinality of the collection $\{\cC^k_{\eta,\gamma^j}\}$
goes to infinity $\cC^k_{\eta,\gamma^j}$ as $j\to \infty$. However,
by Lemma \ref{l:reduction}, the growth rate is $\leq j^{K(\eta,\tv,n)}$, which is 
slow enough to be negligible for our purposes.  This estimate follows from a
quantitative differentiation argument.

 The criterion for membership in each particular set
 $\cC^k_{\eta,\gamma^j}$ represents one of the
possible behaviors on the scales $1,\gamma, \gamma^2, \ldots,\gamma^j$,
which could cause a point to lie in $\cS^k_{\eta,\gamma^j}$.
for $i\leq j$).
A priori, the number of such different behaviors is  $ 2^j$.
However, as explained above, for any fixed $M^n$, $k$, $\eta$,
only a  small fraction $\leq j^{K(\eta,\tv,n)} \cdot 2^{-j}$ of these can
actually occur.

\begin{proof}[Proof of Theorem \ref{t:S^k_Quantitative_Est}]
Let $\underline{x}\in M^n$ and consider $\cS^k_{\eta,r}\cap
B_{\frac{1}{2}}(\underline{x})$ for some fixed
$\eta>0$ as in (\ref{mainve}).  For $c_0=c_0(n)>1$ to be specified below, put
\begin{equation}
\label{e:fixgamma}
\gamma=\gamma(\eta)= c_0^{-\frac{2}{\eta}}\, .
\end{equation}

\begin{lemma}
\label{l:reduction}
There exists $c_1=c_1(n)\geq c_0$,
$K=K(n,\tv,\gamma)$,
$Q=Q(n,\tv,\gamma)=K+n$,
such that for every $j\in\dZ_+$:
\begin{enumerate}
\item The set $\cS^k_{\eta,\gamma^j}\cap B_1(\ux)$ is contained in the union of at most
$j^K$ nonempty sets $\cC^k_{\eta,\gamma^j}$.
\item Each set $\cC^k_{\eta,\gamma^j}$  is the union of at most
$(c_1\gamma^{-n})^{Q}\cdot (c_0\gamma^{-k})^{j-Q}$ balls of radius $\gamma^j$.
\end{enumerate}
\end{lemma}
\vskip2mm

Let us provisionally assume Lemma \ref{l:reduction}. Then by volume comparision, we
have
${\rm Vol}(B_{\gamma^j}(x))\leq c_2(n)\gamma^{jn}$, which together with
$$c_0^j= (\gamma^j)^{-\frac{\eta}{2}}\, ,$$
$$j^K\leq c(n,\tv,\gamma)(\gamma^j)^{-\frac{\eta}{2}}\, ,$$
 gives
\begin{equation}
\label{e:main}
\begin{aligned}
{\rm Vol}(\cS^k_{\eta,\gamma^j}\cap B_1(\underline{x}))
&\leq j^K\cdot \left[  (c_1\gamma^{-n})^{Q} \cdot  (c_0\gamma^{-k})^{j-Q} \right]
 \cdot c_2\cdot(\gamma^j)^n\\
&\leq \underline{c}(n.\tv,\gamma) \cdot j^K\cdot c_0^{j}\cdot(\gamma^j)^{n-k}\\
&\leq \underline{c}(n,\tv,\gamma)\cdot(\gamma^j)^{n-k-\eta}\, ,
\end{aligned}
\end{equation}
where $\underline{c}(n,\tv,\gamma)=(c_1(n)/c_0(n))^Q\cdot c_2(n)\cdot\gamma^{-(n-k)Q}$.
 From the above, for all $r\leq1$, we get (\ref{mainve}) i.e.
$$
 \begin{aligned}
 \Vol(\cS^k_{\eta,r}\cap B_1(\underline{x}))
&\leq \gamma^{-1}\cdot \underline{c}(n,\tv,\gamma)\cdot r^{n-k-\eta}\\
& \leq c(n,{\tv},\eta) r^{n-k-\eta}\, .
\end{aligned}
$$
Therefore, modulo the proof of Lemma \ref{l:reduction}, we get
Theorem \ref{t:S^k_Quantitative_Est}.
\end{proof}

\begin{proof}[Proof of Lemma \ref{l:reduction}]
The sets $\cC^k_{\eta,\gamma^j}$ will be indexed as follows. Consider the
 set of \hbox{$j$-tuples} $T^j$
whose each of whose entries is either $0$, $1$. Denote the number of
entries equal to $1$ by $|T^j|$.
We are going to show the existence of $K=K(n,\tv,\gamma)\in\dZ_+$ (as above) such that
every $\cC^k_{\eta,\gamma^j}$ corresponds to some unique $T^j$
with $|T^j|\leq K$.  We denote this set by $\cC^k_{\eta,\gamma^j}(T^j)$.
Since the number of $T^j$ with $|T^j|\leq K$ is at most
\begin{equation}
\label{e:possibilities}
 {j\choose K} \leq j^{K}\,  ,
\end{equation}
the cardinality of $\{\cC^k_{\eta,\gamma^j}(T^j)\}$ is at most $ j^K$.

In order to specify  the  correspondence $T^j\to \cC^k_{\eta,j}(T^j)$, we need a
quantity
we call the {\it $t$-metric nonconicality} $\cN_t(B_r(x))\geq 0$ of a ball $B_r(x)$.
As in Section \ref{s:sr}, let $C(Z)$ denote the metric cone on $Z$ with vertex
$z^*$. Let $t\geq 1$, then we say $\cN_t(B_r(x))\leq \epsilon$ if there exists $C(Z)$
such that
\begin{equation}
\label{e:gde}
d_{GH}(B_{t r}(x),B_{t r}(z^*))\leq \epsilon r \,  .
\end{equation}
We put
\begin{equation}
\label{e:HL}
\begin{aligned}
H_{t,r,\epsilon} & =\{x\in B_1(\ux)\, |\,   \cN_t(B_r(x))\geq \epsilon\}\, , \\
L_{t,r,\epsilon} &=\{x\in B_1(\ux)\, |\,   \cN_t(B_r(x))< \epsilon\}\,  .
\end{aligned}
\end{equation}
Eventually, we will fix $\epsilon=\epsilon(n,\gamma)$,
the value in Lemma \ref{l:splitting} below.

To each $x$ we associate a $j$-tuple $T^j(x)$. For all $i\leq j$, by definition, the
$i$-th entry of
$T^j(x)$ is $1$ if $x\in H_{\gamma^{-n},\gamma^i,\epsilon}$ and 0 if $x\in
L_{\gamma^{-n},\gamma^i,\epsilon}$.
Then for each $j$-tuple $T^j$ define
$$
E_{T^j}=\{x\in B_1(\ux)\, |\, T^j(x)=T^j \} \, .
$$

Below we will show that if $E_{T^j}$ is nonempty then
\begin{equation}
\label{e:ne}
|T^j|< K(n,\tv,\epsilon) \qquad({\rm if} \,\,E_{T^j}\ne\emptyset)\, .
\end{equation}
Because the sets $\cC^k_{\eta,\gamma^{j-1}}(T^{j-1})$ are indexed by
such tuples, (\ref{e:ne}), together with  (\ref{e:possibilities}), finishes item 1.
of Lemma \ref{l:reduction}.

Let $T^{j-1}$ be obtained from $T^j$ by dropping the last entry.
 Assume that the nonempty subset
$\cC^k_{\eta,\gamma^{j-1}}(T^{j-1})$ (which is a union of balls of radius
$\gamma^{j-1}$) has been defined and satisfies item 2. of the Claim and
 $\cC^k_{\eta,\gamma^{j-1}}(T^{j-1})\supset \cS^k_{\eta,\gamma^j}\cap  E_{T^j}$.
For each ball
$B_{\gamma^{j-1}}(x)$ of $\cC^k_{\eta,\gamma^{j-1}}(T^{j-1})$, take a minimal
covering of
$B_{\gamma^{j-1}}(x)\cap \cS^k_{\eta,\gamma^j}\cap  E_{T^j}$ by balls of radius
$\gamma^j$
with centers in $B_{\gamma^{j-1}}(x)\cap \cS^k_{\eta,\gamma^j}\cap  E_{T^j}$. Define
the union of
all balls so obtained to be $\cC^k_{\eta,\gamma^j}(T^{j})$, {\it provided it is
nonempty}.

 Since $\gamma^j/\gamma^{j-1}=\gamma$, from the lower Ricci curvature bound
(\ref{con:Lower_Ricci}) and relative volume comparison, it is
clear that for each $B_{\gamma^{j-1}}(x)$ as above, the associated minimal covering
has at most
$c_1(n)\gamma^{-n}$ balls. (This is the  $c_1=c_1(n)$ appearing in (\ref{e:main}).)
However,
when $j>n$ and the $j$-entry of $T^j$ is $0$ we use instead the following lemma,
whose proof
will be given in Section \ref{s:sl}.
\begin{lemma}[Covering lemma]
\label{l:splitting}
There exists $\epsilon=\epsilon(n,\gamma)$, such that if
$\cN_{\gamma^{-n}}(B_{\gamma^{j-1}}(x))\leq \epsilon$ and $B_{\gamma^{j-1}}(x)$ is a
ball of
$\cC^k_{\eta,\gamma^{j-1}}(T^{j-1})$, then the number of balls  in the minimal
covering of
$B_{\gamma^{j-1}}(x) \cap \cS^k_{\eta,\gamma^j}\cap
L_{\gamma^{-n},\gamma^j,\epsilon}$ is
$\leq c_0\gamma^{-k}$.
\end{lemma}

\begin{remark}
 In order to apply Lemma \ref{l:splitting}, we need $j>n$.
This explains the appearance of the quantity,  $Q=K+n$  in the
statement of Lemma \ref{l:reduction}.
\end{remark}

\begin{remark}
Lemma \ref{l:splitting} can be viewed as the quantitative analog of the density
argument in the proof that $\dim \cS^k\leq k$.
 Its  proof is a direct consequence of Corollary \ref{c:inductive} of
Lemma \ref{l:qs} (the cone-splitting lemma).
Corollary \ref{c:inductive} provides the quantitative analog of the application of the
splitting theorem  in the proof that $\dim \cS^k\leq k$; see
Section \ref{s:sl}.
 \end{remark}

 Assuming Lemma \ref{l:splitting},
an obvious induction argument yields the bound on the number of
balls of $\cC^k_{\eta,\gamma^j}$ appearing in item 2.
of Lemma \ref{l:reduction}. The factor with exponent $Q$ in item 2. arises
from the (at most $Q$) scales on which the hypothesis of Lemma \ref{l:splitting}
is not satisfied and we are forced to use the standard covering
by at most $c_1\gamma^{-n}$ balls. The factor with exponent $j-Q$
arises from the remaining scales on which we can cover by at most
$c_0\gamma^k$ balls as guaranteed by Lemma \ref{l:splitting}.

We close this section by verifying (\ref{e:ne})
which, as previously noted, suffices to verify item 1. of Lemma \ref{l:reduction}.

Let the notation be as in (\ref{con:Noncollapsed}). For $r>0$,  we consider the
volume ratio
\begin{equation}
\label{d:V_ratio}
\cV_r(x)= \frac{\Vol(B_r(x))}{\Vol_{-1}(r)}\, \downarrow\, .
\end{equation}
The fact that $\cV_r(x)$ is a nonincreasing function of $r$ is just the
Bishop-Gromov inequality.

 For $t>s$, Define the  {\it $(t,s)$-volume energy} $\cW_{t,s}(x)$ by
$$
\cW_{t,s}(x)=
\log\frac{\cV_s(x)}{\cV_{t}(x)}\geq 0\, .
$$
Note that if $s_1\geq t_2$,
 then
\begin{equation}
\label{e:additivity}
\cW_{t_1,s_2}(x)\geq \cW_{t_1,s_1}(x)+\cW_{t_2,s_2}(x)\, ,
\end{equation}
with equality if $t_2=s_1$.
Let  $(s_i,t_i)$ denote a possibly
infinite sequence of intervals with $s_i\geq t_{i+1}$ and $t_1=1$.

Since $\lim_{r\to 0}\log \cV_r(x)=0$ and the $\tv$-noncollapsing
assumption (\ref{con:Noncollapsed}) holds,
 by using (\ref{e:additivity}) together with induction and passing to the limit, we get
\begin{equation}
\label{e:eb}
\log\frac{1}{\tv}\geq\log\frac{1}{\cV_{1}(x)}
\geq \cW_{t_1,s_1} + \cW_{t_2,s_2}+\cdots \, .
\end{equation}
where the terms on the right-hand side are all nonnegative.

 Fix $\delta>0$ and let $N$ denote the number of $i$ such that
$$
\cW_{\gamma^{i-n},\gamma^i}>\delta\, .
$$
Then
\begin{equation}
\label{e:Nb}
N\leq (n+1)\cdot\delta^{-1}\cdot \log\frac{1}{\tv}\, .
\end{equation}
Otherwise, there would be at least $\delta^{-1}\cdot \log\frac{1}{\tv}$ {\it
disjoint} closed  intervals of
 the form $[\gamma^i,\gamma^{i-n}]$ with
$\cW_{\gamma^{i-n}, \gamma^{i}}>\delta$, contradicting (\ref{e:eb}).

Let $\epsilon=\epsilon(n,\gamma)$ be as in Lemma \ref{l:splitting}. The ``almost
volume cone implies almost
metric cone'' theorem of \cite{cheegercoldingalmostrigidity} implies the existence
of $\delta=\delta (\epsilon)$
such that if $\cW_{\gamma^{i-n}, \gamma^{i}}\leq \delta$ then
$\cN_{\gamma^{-n}}(B_{\gamma^i}(x))\leq \epsilon\gamma^i$, i.e.
$x\in L_{\gamma^{-n},\gamma^i,\epsilon}$.
This gives (\ref{e:ne}), which completes the proof of Lemma \ref{l:splitting},
modulo that  of Lemma \ref{l:reduction}.
 \end{proof}

\begin{remark}
Clearly, (\ref{e:ne}) is the quantitative version of the fact that for noncollapsed
limit spaces tangent cones are metric
cones; compare the proof of the inequality, $\dim\cS^k$, which was recalled at the
begining of this section. As previously indicated,
relation (\ref{e:ne}) and its proof
provide an instance of {\it quantitative differentiation} in the sense of Section 14 of
 \cite{CKN}.
\end{remark}

\section{Proof of the covering lemma via the cone-splitting lemma}
\label{s:sl}


Assume that the cone $\dR^\ell\times C(\overline{Z})$ is a Gromov-Hausdorff limit space
with the lower bound on Ricci curvature tending to zero.  Suppose in addition that
there exists
 $y'\not\in \dR^\ell\times \{\overline{z}^*\}$, a
cone $C(\hat{Z})$ and an isometry $I:\dR^\ell\times C(\overline{Z})\to C(\hat{Z})$
 with $I(y')=\hat{z}^*$.  Then $\dR^\ell\times C(\overline{Z})$
 is isometric to a cone $\dR^{\ell+1}\times C(\tilde Z)$.
 This follows because if both
 $\overline{z}^*$ and $y'$ are vertices of cone structures then it is virtually
immediate that
there must be a line which passes through these points. Therefore, the result
follows from
 the splitting theorem; compare the discussion of cone-splitting in Section
\ref{s:outline}.

We continue to denote by $T_s(\,\cdot\,)$ the $s$-tubular neighborhood. Recall that
 $L_{t,r,\epsilon}$ is defined in (\ref{e:HL}).
The above, together with an obvious compactness argument (and rescaling) yields the
following.

\begin{lemma}[Cone-splitting lemma]
\label{l:qs}
For all $\gamma,\tau,\psi >0$ there exists
$0<\epsilon=\epsilon(n,\gamma,\tau,\psi)<\psi$,
$0<\theta= \theta(n,\gamma,\tau,\psi)$, such that
the following holds. Let $r\leq \theta$ and assume that for some cone
$\dR^\ell \times C(Z)$ there is \hbox{$\epsilon r$-Gromov-Hausdorff} equivalence
$$
F:  B_{\gamma^{-1}r}((\underline{0},z^*))\to B_{\gamma^{-1}r}(x)\, .
$$
 If there exists
$$
x'\in B_r(x)\cap L_{\gamma^{-1},r,\epsilon}\, ,
$$
 with
$$
x'\not\in T_{\tau r}(F(\dR^\ell\times \{z^*\}))\cap B_r(x)\, ,
$$
then for some cone $\dR^{\ell+1}\times C(\tilde Z)$,
$$
d_{GH}(B_r(x),B_r((\underline{0},\tilde z^*)))<\psi r\, .
$$
\end{lemma}


\begin{corollary}\label{c:inductive}
For all $\gamma, \tau,\psi >0$ there exists $0<\delta(n,\gamma,\tau,\psi)$
 and $0<\theta(n,\gamma,\tau,\psi)$ such that the following holds.  Let $r\leq
\theta$ and
 $x\in L_{\gamma^{-n},\delta,r}$. Then there exists a cone $\dR^{\ell} \times
C(\tilde{Z})$
with a \hbox{$\psi r$-Gromov-Hausdorff} equivalence
$$
F:  B_{r}((\underline{0},\tilde z^*))\to B_{r}(x)\, ,
$$
such that
$$
L_{\gamma^{-n},\delta,r}\cap B_r(x)\subseteq T_{\tau r}(F(\dR^\ell\times \{\tilde
z^*\}))\, .
$$
\end{corollary}

\begin{proof}
 For $\epsilon(n,\gamma,\tau,\psi)$ as in Lemma \ref{l:qs}, inductively define
$\epsilon^{[n-i]}=\epsilon\circ\epsilon\circ\cdots\circ\epsilon(n,\gamma^{-n},\tau,\psi)$
\hbox{($i$ factors in the composition)}. Then
$\epsilon^{[0]}<\epsilon^{[1]}<\cdots<\epsilon$.
Put $\delta=\epsilon^{[0]}$.
Since by assumption, $x\in L_{\gamma^{-n},\delta,r}$,  there exists a largest $\ell$
such
that  for some cone $\dR^\ell\times C(\tilde{Z})$,
there is an \hbox{$\epsilon^{[n-\ell]} r$-Gromov-Hausdorff} equivalence
$F: B_{\gamma^{-(n-\ell)}r}((\underline{0},z^*) \to B_{\gamma^{-(n-\ell)}r}(x)$.
To see that the conclusion holds for this value of $\ell$, apply
Lemma \ref{l:qs} with the replacements: $r\to \gamma^{-(n-\ell-1)}r$,
$\tau=\gamma^{-(n-\ell-1)}\tau$,
$\epsilon\to \epsilon^{[\ell]}$, $\psi\to\epsilon^{[\ell+1]}$.
\end{proof}

\begin{proof}[Proof of Lemma \ref{l:splitting}]
Let $B_{\gamma^{j-1}}(x)$ be as in Lemma \ref{l:splitting}.  Since by assumption,
$x\in \cS^k_{\eta,\gamma^i}\cap L_{\gamma^{-n},\gamma^j,\epsilon}$
no  cone as in (\ref{e:gde}) with $t=\gamma^{-n}\cdot\gamma^{j-1}$ can split off a
factor $\R^{k+1}$
isometrically.
By applying Corollary \ref{c:inductive} with $r=\gamma^{j-1}$,
 $\psi=\frac{1}{10}\gamma$ it follows that for some $\ell\leq k$ and $F$
as in the corollary, we have
$$
B_{\gamma^{j-1}}(x)\cap \cS^k_{\eta,\gamma^i}\cap L_{\gamma^{-n},\gamma^j,\epsilon}
\subset F(T_{\frac{1}{10}\gamma^j}(\dR^\ell\times \{\tilde z^*\}))
\cap B_{\gamma^{j-1}}(x)\, .
$$
Clearly, this suffices to complete the proof.
\end{proof}

\section{Curvature estimates absent a priori integral bounds}
\label{s:Effective_Curvature_Estimates1}

In this short section we prove Theorem  \ref{t:Effective_Curvature_Estimate1}.
Recall that the assumptions are that
$(M^n,g)$ is an Einstein manifold which satisfies
the $\tv$-noncollapsing  condition (\ref{con:Noncollapsed}) and
 the bound (\ref{con:Bounded_Ricci}) on the Einstein constant.
Item 1. pertains to the real case and item 2. to the K\"ahler case.
  The curvature estimates of Theorem \ref{t:Effective_Curvature_Estimate1} follow by
combining the
geometric $\epsilon$-regularity theorems of \cite{cheegercoldingonthestruct1}
 and \cite{MR1978491} with Theorem \ref{t:S^k_Quantitative_Est}.
The proofs of these theorems rely on an $\epsilon$-regularity
theorem of  Anderson; see  \cite{MR1074481}.
We now recall the statements.

Let $(\underline{0},z^*)$ denote the vertex of the cone
$\dR^\ell\times C(Z)$.
Assume that $(M^n,g)$ is an Einstein manifold which satisfies
the $\tv$-noncollapsing  condition (\ref{con:Noncollapsed}) and
 the bound (\ref{con:Bounded_Ricci}) on the Einstein constant.
In our language, the $\epsilon$-regularity theorem of
\cite{cheegercoldingonthestruct1}, which does not
assume the K\"ahler condition, asserts that
there exists $\epsilon_0(n,\tv)>0$
 such that if
 $$
d_{GH}(B_r(x),B_r((\underline{0},z^*)))\leq \epsilon_0 r\, , \qquad (\ell>n-2)\, ,
$$
then on $B_{\frac{1}{2}r}(x)$ there exists a harmonic coordinate system
in which the $g_{ij}$ and $g^{ij}$ have definite $C^k$ bounds, for all $k$.  In
particular, the $C^2$-harmonic radius satisfies
$r_{har}(x)\geq c(n) r$; $\clubsuit$ see Definition \ref{hrdef}.

By \cite{MR1978491},  in
 the K\"ahler-Einstein case,  the same
conclusion holds if $\ell>n-4$.
(Conjecturally, the K\"ahler condition can be dropped.)

\begin{proof}[Proof of Theorem \ref{t:Effective_Curvature_Estimate1}]
Since the arguments are mutadis mutandis the same for items 1. and 2.
of Theorem \ref{t:Effective_Curvature_Estimate1}, we will just prove
item 1.
In this case, by the $\epsilon$-regularity theorem, for all $\eta\leq \epsilon_0$,
$$
\cB_r\cap B_{\frac{1}{2}}(\ux)\subseteq T_{r}(\cS^{n-2}_{\eta,Cr})\cap
B_{\frac{1}{2}}(\ux)\, .
$$
Thus, by Theorem
 \ref{t:S^k_Quantitative_Est}, we have
$$
\begin{aligned}
\Vol(\cB_r\cap B_{\frac{1}{2}}(x))& \leq \Vol(T_{Cr}(\cS^{n-2}_{\eta,Cr})
                                                         \cap B_{\frac{1}{2}}(x))\\
                                                     &\leq C(n,\tv,\eta)r^{2-\eta}\, ,
\end{aligned}
$$
 which
 completes the proof.
\end{proof}

\section{Curvature estimates given a priori integral bounds}
\label{s:Effective_Curvature_Estimates2}

In this section we prove Theorem
\ref{t:Effective_Curvature_Estimate2}.

The proof  uses the following corollary of Theorem \ref{t:S^k_Quantitative_Est}.
 For $r_1<r_2$, put
\begin{equation}
\label{e:genS}
\cS^k_{\eta,r_1,r_2}:=\{x  \, |\,
d_{GH}(B_s(x),B_s((\underline{0},z^*)))
\geq\eta s, \,\,{\rm for\,\, all}\,\, \dR^{k+1}\times C(Z)\, \, {\rm and\,\,
all}\,\, r_1\leq s\leq  r_2\}\, .
\end{equation}
\begin{corollary}
\label{rescaled}
\begin{equation}
\label{e:rescaledS}
\begin{aligned}
\Vol(\cS^k_{\eta,r_1,r_2}\cap B_{r_1}(x))&\leq
c(n,\tv,\eta)(r_2^{-1}r_1)^{-(k+\eta)}\cdot r_1^{n}\\
                                                                               &=c(n,\tv,\eta)r_1^{n-k-\eta}\cdot
r_2^{(k+\eta)}
\end{aligned}
\end{equation}
\end{corollary}
\begin{proof}
 Let $\hat B_{r_2}(x)$ denote  the ball of radius $\frac{1}{2}$ obtained by
rescaling the
metric on $B_{r_2}(x)$ by a factor $\frac{1}{2}\cdot r^{-1}_2$ and let $
\hat\cS^k_{\eta,r}$
denote $\cS^k_{\eta,r}$ for the rescaled metric.
Then
$$
\cS^k_{\eta,r_1,r_2}\cap B_{r_2}(x)= \hat\cS^k_{\eta,\frac{1}{2}r_2^{-1}r_1}
  \cap\hat B_{\frac{1}{2}r_2^{-1}r_1}(x)\,.
$$
If we apply Theorem \ref{t:S^k_Quantitative_Est} in the rescaled situation and
interpret the
 conclusion for the original metric, we get (\ref{e:rescaledS}).
\end{proof}

Recall that in addition to (\ref{con:Noncollapsed}), (\ref{con:Bounded_Ricci}),
and the assumption that $(M^n,g)$ is Einstein, we assume
the $L_p$ curvature bound (\ref{con:L^p_Curvature_Bound}).

The proof of Theorem \ref{t:Effective_Curvature_Estimate1} also uses
 the $\epsilon$-regularity theorems of \cite{MR1937830} ($p=2$),
\cite{MR1978491} ($p\geq 2$) and  Theorem \ref{t:S^k_Quantitative_Est} for the case
$k=n-2p-1$.
We now recall the statement from \cite{MR1978491}.

As usual, $(\underline{0},z^*)$ denotes the vertex of the cone
$\dR^\ell\times C(Z)$.  Assume that $(M^n,g)$ is a K\"ahler-Einstein manifold which
satisfies
the $\tv$-noncollapsing  condition (\ref{con:Noncollapsed}) and
 the bound (\ref{con:Bounded_Ricci}) on the Einstein constant.
Then there exists $\epsilon_0(n,\tv,p)>0$, $\eta_0(n,\tv,p)>0$ such that if
\begin{equation}\label{e:cc}
d_{GH}(B_{r}(x),B_{r}((\underline{0},z^*)))<\eta_0 r\qquad (\ell\geq n-2p)\, ,
\end{equation}
\begin{equation}\label{e:sic}
r^{2p}\fint_{B_{r}(x)}|Rm|^{p}\leq \epsilon_0\, ,
\end{equation}
then on $B_{\frac{1}{2}r}(x)$ there exists a harmonic coordinate system
in which the $g_{ij}$ and $g^{ij}$ have definite $C^k$ bounds, for all $k$.  In
particular,
$r_{har}(x)\geq c(n) r$.

\begin{proof}[Proof of Theorem \ref{t:Effective_Curvature_Estimate2}]
Note that since the $\epsilon$-regularity theorem requires that two independent
conditions hold simultaneously, we must control the collection
of balls on which either one of them fails to hold.

 Fix $\epsilon_0$ as above and let $\cD_{\epsilon_0,r}$ denote the union
of the balls  $B_r(x)$ with $x\in B_{\frac{1}{2}}(\ux)$,
  for which (\ref{e:sic}) fails to hold.
 By a standard covering argument
 it follows from the $L_p$ curvature bound (\ref{con:L^p_Curvature_Bound})
that  $\cD_{\epsilon_0,r}\cap B_{\frac{1}{2}}(\underline{x})$
can be covered by a collection of balls $\{B_{r}(x_i)\}$ such that we have
\begin{equation}
\label{e:vs}
\Vol(\cD_{\epsilon_0,r}\cap B_{\frac{1}{2}}(\underline{x}))\leq
\sum_i \Vol(B_{r} (x_i))\leq c(n)C\epsilon_0^{-1}\cdot r^{2p}\, .
\end{equation}

 In particular, for $\gamma$ as in Section \ref{s:ip}, $\eta=\eta_0<1$ and $k=n-2p-1$,
by applying Corollary \ref{rescaled}
to each the balls $B_r(x_i)$ whose union covers
$\cD_{\epsilon_0,\gamma^i}$ and summing the resulting
estimates we get  for all $1\leq i<j$,
\begin{equation}
\label{e:individual}
\Vol(\cS^{n-2p-1}_{\eta,\gamma^j,\gamma^{i}}\cap \cD_{\epsilon_0,\gamma^{i-1}})
 \leq c(n,\tv,\eta_0,C)(\gamma^j)^{2p+1-\eta_0}\cdot (\gamma^{i})^{1+\eta_0}\,  .
\end{equation}
Summing  (\ref{e:individual}) over $1\leq i\leq j$
and bounding  the right-hand side in terms of  the   geometric series with ratio
$\gamma^{1+\eta_0}$ gives
\begin{equation}
\label{e:Db}
\sum_{1\leq i\leq j}\Vol(\cS^{n-2p-1}_{\eta_0,\gamma^j,\gamma^i}\cap
\cD_{\epsilon_0,\gamma^{i-1}})
\leq c(n,\tv,\eta_0,C,p)(\gamma^j)^{2p+1-\eta_0}\, .
\end{equation}

We claim that
\begin{equation}
\label{e:sscover}
\cB_{\gamma^j}\cap B_{\frac{1}{2}}(\ux)\subset
\left(\cS^{n-2p-1}_{\eta_0,\gamma^j}\cap B_{\frac{1}{2}}(\ux)\right) \cup
\left(\bigcup_{1\leq i\leq j}\cS^{n-2p-1}_{\eta_0,\gamma^j,\gamma^i}\cap
\cD_{\epsilon_0,\gamma^{i-1}}\right)
\cup \cD_{\epsilon_0,\gamma^j}\, .
\end{equation}
This will suffice to complete the proof of Theorem
\ref{t:Effective_Curvature_Estimate2}
for the case $r=\gamma^j$,
since by (\ref{e:Db}), together with (\ref{e:vs}) for $r=\gamma^j$ and
Theorem \ref{t:S^k_Quantitative_Est} for $r=\gamma^j$, it follows that
the volume of the set on the right-hand side is 
$\leq c(n,tv,\eta_0,\epsilon_0,p,C)(\gamma^j)^{2p}$. As in the proof of
Theorem \ref{t:S^k_Quantitative_Est} the general case  follows directly from the
case $r=\gamma^j$.

Let $A'$  denote the complement of $A$.
To establish the claim, we note that the complement of the set on right-hand side of
(\ref{e:sscover}) is equal to
$$
 ((\cS^{n-2p-1}_{\eta_0,\gamma^j})' \cup B_{\frac{1}{2}}(\ux)')\cap
\left(\bigcap_{0\leq i\leq j}(\cS^{n-2p-1}_{\eta_0,\gamma^j,\gamma^i})'\cup
(\cD_{\epsilon_0,\gamma^{i-1}})'\right)
\cap  (\cD_{\epsilon_0,\gamma^j})'
\, .
$$
By expanding out  and
 dropping the terms which start with
$B_{\frac{1}{2}}(\ux)'$, we obtain an expression
that is a union  of terms, each of which is  of the form
\begin{equation}
\label{e:transition}
\begin{aligned}
(\cS^{n-2p-1}_{\eta_0,\gamma^j})' \cap \cdots\cap
&(\cS^{n-2p-1}_{\eta_0,\gamma^j,\gamma^i})'\cap (\cD_{\epsilon_0,\gamma^{i}})'
\cap \left(\bigcap_{i<\ell\leq j}(\cD_{\epsilon_0,\gamma^{\ell}})'\right)
\cap   (\cD_{\epsilon_0,\gamma^j})'\\
&{}\\
\subset\, &(\cS^{n-2p-1}_{\eta_0,\gamma^j,\gamma^i})'\cap
(\cD_{\epsilon_0,\gamma^{i}})'
\cap(\cD_{\epsilon_0,\gamma^{{i+1}}})'\cap \cdots\cap  (\cD_{\epsilon_0,\gamma^j})'\, .
\end{aligned}
\end{equation}
for some $i$ with $1\leq i\leq j$. (The terms represented by the dots
can be either $\cS$'s or $\cD$'s.)
By  (\ref{e:cc}), (\ref{e:sic}), the set
 on the second line of (\ref{e:transition}) satisfies the hypothesis of the
$\epsilon$-regularity theorem of
\cite{MR1978491}, so this completes the proof.
\end{proof}


\bibliography{rgcn}

\def\cprime{$'$} \def\cprime{$'$}
\begin{thebibliography}{CCT02}

\bibitem[And90]{MR1074481}
M.~T. Anderson.
\newblock Convergence and rigidity of manifolds under {R}icci curvature bounds.
\newblock {\em Invent. Math.}, 102(2):429--445, 1990.

\bibitem[CC96]{cheegercoldingalmostrigidity}
J.~Cheeger and T.~H. Colding.
\newblock Lower bounds on {R}icci curvature and the almost rigidity of warped
  products.
\newblock {\em Ann. of Math. (2)}, 144(1):189--237, 1996.

\bibitem[CC97]{cheegercoldingonthestruct1}
J.~Cheeger and T.~H. Colding.
\newblock On the structure of spaces with {R}icci curvature bounded below. {I}.
\newblock {\em J. Differential Geom.}, 46(3):406--480, 1997.

\bibitem[CCT02]{MR1937830}
J.~Cheeger, T.~H. Colding, and G.~Tian.
\newblock On the singularities of spaces with bounded {R}icci curvature.
\newblock {\em Geom. Funct. Anal.}, 12(5):873--914, 2002.

\bibitem[CD11a]{chendonaldson2}
X.-X. Chen and S.~Donaldson.
\newblock Volume estimates for {K}\"ahler {E}instein metrics and rigidity of
  complex structures.
\newblock {\em ArXiv:math/1104.43310v1}, 2011.

\bibitem[CD11b]{chendonaldson1}
X.-X. Chen and S.~Donaldson.
\newblock Volume estimates for {K}\"ahler {E}instein metrics: the three
  dimensional case.
\newblock {\em ArXiv:math/1104.0270v2}, 2011.

\bibitem[Che03]{MR1978491}
J.~Cheeger.
\newblock Integral bounds on curvature elliptic estimates and rectifiability of
  singular sets.
\newblock {\em Geom. Funct. Anal.}, 13(1):20--72, 2003.

\bibitem[CKNar]{CKN}
J.~Cheeger, B.~Kleiner, and A.~Naor.
\newblock Compression bounds for {L}ipschitz maps from the {H}eisenberg group
  to ${L}_1$.
\newblock {\em Acta. Math.}, (to appear).

\bibitem[CN11]{cnminhar}
J~Cheeger and A.~Naber.
\newblock Quantitative stratification and the regularity of harmonic maps and
  minimal currents.
\newblock {\em ArXiv:math/1107.1197}, 2011.

\bibitem[CS09]{chakrabartishaw}
D.~Chakrabarti and M.-C. Shaw.
\newblock The {C}auchy-{R}iemann equations on product domains.
\newblock {\em ArXiv:math/09011.0103}, 2009.

\end{thebibliography}
\bibliographystyle{alpha}

\end{document}